\documentclass[12pt]{amsart}
\usepackage{latexsym}
\usepackage{amsthm}
\usepackage{amsmath}
\usepackage{amsfonts}
\usepackage{amssymb}
\usepackage[dvips]{graphicx}

\input xy
\xyoption{all}
\SelectTips{cm}{}

\newtheorem{theorem}{Theorem}[section]
\newtheorem{lemma}[theorem]{Lemma}

\newtheorem{corollary}[theorem]{Corollary}

\newcommand\calC{\mathcal{C}}
\newcommand\calD{\mathcal{D}}
\newcommand\calM{\mathcal{M}}

\newcommand\calG{\mathcal{I}}
\newcommand\calJ{\mathcal{J}}
\newcommand\calK{\mathcal{K}}
\newcommand\Acyclic{{\mathcal{A}}}

\newcommand\Ohkawa{{\mathcal{T}}}
\newcommand\Trivializer{{T}}
\newcommand\colim{{\rm colim}}
\newcommand\Ho{{\rm Ho}}
\newcommand\Ch{{\rm Ch}}

\newcommand\ZZ{\mathbb{Z}}
\newcommand\NN{\mathbb{N}}
\newcommand\SH{{\mathcal{SH}}(S)}
\newcommand\Mot{{\mathcal M}}

\date{\today}

\numberwithin{equation}{section}

\begin{document}
\title{A generalization of Ohkawa's theorem}
\author[C. Casacuberta]{Carles Casacuberta}
\address{Institut de Matem\`atica, Universitat de Barcelona, Gran Via de les Corts Catalanes 585, 08007 Barcelona, Spain}
\email{carles.casacuberta@ub.edu}
\author[Javier J. Guti\'errez]{Javier J. Guti\'errez}
\address{Department of Algebra and Topology,
Radboud Universiteit \newline Nijmegen, Heyendaalseweg 135,
6525 AJ Nijmegen, The Netherlands}
\email{j.gutierrez@math.ru.nl}
\author[J. Rosick\'{y}]{Ji\v{r}\'{\i} Rosick\'{y}}
\address{Department of Mathematics and Statistics, Masaryk University, Faculty of Science, Kotl\'{a}\v{r}sk\'{a}~2, 61137 Brno, Czech Republic}
\email{rosicky@math.muni.cz}
\thanks{The two first-named authors were supported by the Spanish Ministry of Economy and Competitiveness under grant MTM2010-15831 and by the Generalitat de Catalunya as members of the team 2009~SGR~119. 
The third-named author was supported by GA\v{C}R 201/11/0528.}

\begin{abstract}
A theorem due to Ohkawa states that the collection of Bousfield
equivalence classes of spectra is a set.
We extend this result to arbitrary combinatorial model categories.
\end{abstract}

\maketitle

\section*{Introduction}
Ohkawa proved in \cite{O} that the homotopy category of spectra has only a set 
(that is, not a proper class) of distinct homological acyclic classes. The \emph{homological acyclic class} or \emph{Bousfield class} $\langle E\rangle$ of a spectrum $E$ consists of all $E_*$\nobreakdash-acyclic spectra, where $E_*$ is the reduced homology theory represented by~$E$. In other words, $\langle E\rangle$ is the collection of spectra $X$ such that $E\wedge X=0$ in the homotopy category. 
The original source of this terminology is~\cite{B}.

Bousfield classes are closely related with localizations. The earliest form of localization in homotopy theory \cite{Su} was a technique to split homotopy types into their $p$\nobreakdash-primary components for all primes~$p$, thereby introducing the use of Hasse-principle methods in topology, both for spaces and for spectra. A~decade later, it was discovered that every $p$\nobreakdash-local spectrum could be further resolved into \emph{$v_n$\nobreakdash-periodic} components for $n\ge 0$. The resulting \emph{chromatic towers} and their associated spectral sequences became major tools to compute stable homotopy groups~\cite{R}.

All these are special cases of homological localizations. For each reduced homology theory $E_*$ defined on spaces or spectra there is an \emph{$E_*$\nobreakdash-localization functor} \cite{B}, which transforms the $E_*$\nobreakdash-equivalences 
(that is, maps $X\to Y$ inducing isomorphisms $E_k(X)\cong E_k(Y)$ for all~$k$) into homotopy equivalences in a universal way.
Localization at a prime $p$ is obtained by letting $E_*$ be ordinary homology with $p$\nobreakdash-local coefficients, and the $n$th stage of the chromatic resolution is $E(n)_*$\nobreakdash-localization, where $E(n)=K(0)\vee\cdots\vee K(n)$ is a wedge of Morava $K$\nobreakdash-theories~\cite{JW}.

Two spectra $E$ and $F$ are called \emph{Bousfield equivalent} if $E_*$\nobreakdash-local\-iza\-tion is equivalent to $F_*$\nobreakdash-local\-iza\-tion. This happens precisely when the classes of $E_*$\nobreakdash-acyclic spectra and $F_*$\nobreakdash-acylic spectra coincide, that is, when the Bousfield classes $\langle E\rangle$ and $\langle F\rangle$ are identical.
Thus, according to Ohkawa's theorem, Bousfield equivalence classes of spectra form a set. A~shorter proof of this fact was given by Dwyer and Palmieri in~\cite{DP}, and some consequences were described in~\cite{HP}.

In a different direction, Neeman proved in \cite{N} that Bousfield classes form a set
in the derived category of any commutative Noetherian ring.
In this context, the Bousfield class of a chain complex $A$ is defined as the collection
of chain complexes $X$ such that the derived tensor product $A\otimes X$ is zero.
Dwyer and Palmieri proved the same result in \cite{DP2} for the derived category of a truncated polynomial ring on countably many generators
over a countable field. 
They asked in \cite[Question~5.9]{DP2} if Ohkawa's theorem
is in fact true in the derived category of every commutative ring. This was answered
in the affirmative by Stevenson in \cite{S} and by Iyengar and Krause in~\cite{IK}, and it also follows from the results of the present article.

Both the homotopy category of spectra and the derived category
of a commutative ring are homotopy categories of \emph{combinatorial model categories}, and their tensor product comes
from a closed monoidal structure in the model category. In this article we prove that the collection of Bousfield classes is a set under these general assumptions.
This extends the validity of Ohkawa's theorem, for example, to categories of motivic spaces or motivic spectra over any base scheme~\cite{MV}, and to categories of modules over (ordinary or motivic) ring spectra. Thus, Okhawa's theorem also holds in the derived category of motives over any field $k$ of characteristic zero, since these are modules over a motivic Eilenberg--Mac\,Lane spectrum \cite{RO2}.

Specifically, we show that in every combinatorial model category $\calM$ (neither necessarily stable nor pointed), for every sufficiently large regular cardinal $\lambda$
there is only a set of distinct acyclic classes $\Acyclic(H)$ for functors $H\colon\calM\to\calM$ preserving $\lambda$\nobreakdash-filtered colimits and such that the terminal object of $\calM$ is $H$\nobreakdash-acyclic. 
An object $X$ of $\calM$ is called \emph{$H$\nobreakdash-acyclic} if $HX$
is weakly equivalent to the terminal object, and we denote by $\Acyclic(H)$ the collection of all $H$\nobreakdash-acyclic objects.
If a model category $\calM$ is closed monoidal, combinatorial and pointed, then, since left adjoints preserve all co\-limits
and there are cofibrant replacement functors on $\calM$ preserving $\lambda$\nobreakdash-filtered colimits for sufficiently 
large~$\lambda$, it follows that Bousfield classes in the homotopy category of $\calM$ form a~set.

In contrast with this fact, in the derived category of $\ZZ$ or in the homotopy category of spectra there is a proper class of distinct acyclic classes for nullification functors; see \cite[\S\,8]{St} for terminology and details. Each nullification functor $P_A$ preserves $\lambda$\nobreakdash-filtered colimits for $\lambda$ big enough, although the size of $\lambda$ increases with~$A$.

Our method of proof of Ohkawa's theorem for combinatorial model categories generalizes the argument given in~\cite{DP}.  
A similar argument was used in \cite{S} for compactly generated tensor triangulated categories.
Using a different approach, it was shown in~\cite[Theorem~3.1]{IK}
that every well generated tensor triangulated category has only a set of Bousfield classes.  
This result is consistent with the fact that
homotopy categories of stable combinatorial model categories are well generated.

Nevertheless, we emphasize that Ohkawa's theorem is by far not exclusively a result about triangulated categories.
For example, Corollary~\ref{discrete} below implies that there is only a set of
homological acyclic classes of simplicial sets or motivic spaces for every base scheme, and our proof just relies on the fact that these categories are locally presentable and  homology functors preserve filtered colimits. 

\bigskip

\noindent
\textbf{Acknowledgements}
We are indebted to Fernando Muro for frequent
exchanges of views on this topic, which made us rethink earlier versions of the article. Corollary~3.7 was kindly pointed out by Paul Arne {\O}stv{\ae}r.
We also appreciate input from George Raptis and Greg Stevenson.

\section{Combinatorial model categories}
\label{prelims}

We assume that regular cardinals are infinite.
For a regular cardinal~$\lambda$,
a small category $\calK$ is \emph{$\lambda$\nobreakdash-filtered} if it is nonempty and,
given any set of objects $\{k_i\mid i\in I\}$ where $|I|<\lambda$, there is an object $k$ and a
morphism $k_i\to k$ for each $i\in I$, and, moreover, given any set of parallel arrows
between two fixed objects $\{\alpha_j\colon k\to k' \mid j\in J\}$ where $|J|<\lambda$, there is a morphism
$\gamma\colon k'\to k''$ such that $\gamma\circ\alpha_j$ is the same morphism for all $j\in J$.

An object $X$ of a category $\calC$ is \emph{$\lambda$\nobreakdash-presentable}
if the functor $\calC(X,-)$ from $\calC$ to sets preserves $\lambda$\nobreakdash-filtered colimits.
A~cocomplete category $\calC$ is \emph{locally $\lambda$\nobreakdash-pres\-ent\-able} if the collection of isomorphism classes of $\lambda$\nobreakdash-presentable objects is a set and every object of $\calC$ is a $\lambda$\nobreakdash-filtered colimit of $\lambda$\nobreakdash-presentable objects.
A~category is called \emph{locally presentable} if it is locally $\lambda$\nobreakdash-presentable for
some regular cardinal~$\lambda$.
See \cite[Section~1.B]{AR}, \cite{GU} or \cite{MP} for further information about locally presentable categories.

The essentials of Quillen model categories can be found in \cite{H} or~\cite{Q}.
A model category is \emph{pointed} if it has a zero object, i.e., 
if the initial object and the terminal object are isomorphic.

A~model category $\calM$ is called \emph{combinatorial} if it is cofibrantly generated~\cite{Hi, H} and the underlying category is locally presentable.
Dugger proved in \cite{D} that a model category is combinatorial if and only if it is Quillen equivalent to a left Bousfield localization of a category of diagrams of simplicial sets equipped with the projective model structure. Hence, many model categories of interest in various contexts are combinatorial. Examples relevant to the present article are pointed or unpointed simplicial sets, pointed or unpointed motivic spaces \cite{DRO, MV}, symmetric spectra over simplicial sets \cite[\S\,3.4]{HSS} or over motivic spaces~\cite{J}, module spectra over a ring spectrum \cite[Theorem~4.1]{SS}, and bounded or unbounded chain complexes of modules over a ring \cite[\S\,2.3]{H}.

\begin{lemma}
\label{suff}
If $\calM$ is a combinatorial model category, then for every ordinal $\alpha$ there is a regular cardinal $\lambda>\alpha$ with the following properties:
\begin{itemize}
\item[{\rm (i)}] $\calM$ is locally $\lambda$\nobreakdash-presentable;
\item[{\rm (ii)}] there are sets of generating cofibrations and generating trivial cofibrations in $\calM$ whose domains and codomains are $\lambda$\nobreakdash-pres\-ent\-able;
\item[{\rm (iii)}] there are fibrant and cofibrant replacement functors on $\calM$ that preserve $\lambda$\nobreakdash-filtered colimits;
\item[{\rm (iv)}] the terminal object of $\calM$ is $\lambda$\nobreakdash-presentable.
\end{itemize}
\end{lemma}

\begin{proof}
Take first a regular cardinal $\mu>\alpha$
such that $\calM$ is locally $\mu$\nobreakdash-presentable. This is possible since, by \cite[Theorem~1.20]{AR}, 
if $\calM$ is locally $\nu$\nobreakdash-presentable and $\nu'\ge\nu$ then $\calM$ is also locally $\nu'$\nobreakdash-presentable. 
Next, pick a set $\calG$ of generating cofibrations and a set $\calJ$ of generating trivial cofibrations in $\calM$
and choose a regular cardinal $\lambda\ge\mu$ big enough so that all the domains and codomains
of morphisms in $\calG$ and $\calJ$ are $\lambda$\nobreakdash-presentable, and such that the terminal object of $\calM$
is $\lambda$\nobreakdash-presentable as well.
Such a choice is possible by \cite[Proposition~1.16 and Remark~1.30(1)]{AR}.
Finally, (iii) is a consequence of (i) and~(ii), as shown in \cite[\S 7]{D} or \cite[\S 3]{R1}.
\end{proof}

For a combinatorial model category $\calM$ and a sufficiently big regular cardinal~$\lambda$ (as provided by Lemma~\ref{suff}), we use the term \emph{$\lambda$\nobreakdash-combinatorial structure} on $\calM$ to designate a choice of the following items:
a set $\calM_{\lambda}$ of representatives of isomorphism classes of $\lambda$\nobreakdash-presentable objects, including the terminal object, such that every object of $\calM$ is a $\lambda$\nobreakdash-filtered colimit of objects in~$\calM_{\lambda}$;
a set $\calG$ of generating cofibrations and a set $\calJ$ of generating trivial cofibrations whose domains and codomains are in~$\calM_{\lambda}$; and a fibrant replacement functor and a cofibrant replacement functor both preserving $\lambda$\nobreakdash-filtered colimits.

Suppose that a category $\calC$ is locally $\lambda$\nobreakdash-presentable and its terminal object is $\lambda$\nobreakdash-presentable. Then, if we endow $\calC$ with the \emph{discrete} model structure, where the weak equivalences are the isomorphisms
and all morphisms are fibrations and cofibrations, the resulting model category has a $\lambda$\nobreakdash-combinatorial structure where the set $\calG$ of generating cofibrations is the set of all morphisms between members of the chosen set $\calC_{\lambda}$; cf.\ \cite[Example~4.6]{R2}. Recall that locally presentable categories are cocomplete by definition and they are also complete by \cite[Corollary~1.28]{AR}.

The condition that the terminal object be $\lambda$\nobreakdash-presentable
holds automatically when it is a zero object, but may fail otherwise, as exemplified by the category  ${\rm Set}^I$ of $I$\nobreakdash-sorted sets (i.e., functors $I\to{\rm Set}$), where $I$ is any infinite set. This category is locally $\aleph_0$\nobreakdash-presentable by \cite[Corollary~1.54]{AR}, yet its terminal object is not $\aleph_0$\nobreakdash-presentable.

\section{Main result}
\label{results}

Let $\calM$ be a combinatorial model category and suppose given a $\lambda$\nobreakdash-combinatorial structure on it for a suitable regular cardinal~$\lambda$.
Recall that, if $\calG$ is the given set of generating cofibrations, then a morphism $f\colon X\to Y$ is a trivial fibration in $\calM$ if and only if it has the right
lifting property with respect to all the morphisms in~$\calG$. 

An object $X$ of $\calM$ is called \emph{contractible} if the unique morphism 
from $X$ to the terminal object $*$ is a weak equivalence.
For a functor $H\colon\calM\to\calM$, an object $X$ is called \emph{$H$\nobreakdash-acyclic} if $HX$ is contractible.
We denote by $\Acyclic(H)$ the collection of all $H$\nobreakdash-acyclic objects in~$\calM$.

Given a functor $H\colon\calM\to\calM$ and a
triple $(\sigma,A,f)$ where $\sigma\colon P\to Q$ is in $\calG$ and 
\[
f\colon P\longrightarrow RHA
\] 
is a morphism with $A\in\calM_{\lambda}$, where $R$ is the given fibrant replacement functor, we denote by $\Trivializer_H(\sigma,A,f)$ the set of all 
morphisms $t\colon A\to B$ with $B\in\calM_{\lambda}$ for which there exists a morphism $g\colon Q\to RHB$ such that
$RHt\circ f=g\circ\sigma$:
\[
\xymatrix{
P\ar[d]^-{\sigma} \ar[r]^-{f} & RHA \ar[rr]^{RHt} & & RHB. \\
Q \ar@{.>}[urrr]_g
}
\]
Note that, since the terminal object $*$ is in $\calM_{\lambda}$, if $H(*)$ is contractible then the morphism $A\to *$ is in $\Trivializer_H(\sigma,A,f)$
for every $(\sigma,A,f)$.

Finally, let $\Ohkawa(H)$ denote the set whose elements are all the distinct sets
$\Trivializer_H(\sigma,A,f)$ with $A\in\calM_{\lambda}$, $\sigma\colon P\to Q$ in $\calG$, and $f\colon P\to RHA$. 

\begin{theorem}
\label{mainthm}
Suppose given a $\lambda$\nobreakdash-combinatorial structure on a model category $\calM$ for a regular cardinal~$\lambda$. Let $H_1$ and $H_2$ be endofunctors of $\calM$ that preserve $\lambda$\nobreakdash-filtered colimits. Then, if $\Ohkawa(H_2)\subseteq\Ohkawa(H_1)$ and the terminal object of $\calM$ is $H_2$\nobreakdash-acyclic, it follows that $\Acyclic(H_1)\subseteq\Acyclic(H_2)$.
\end{theorem}

\begin{proof}
Let $X$ be $H_1$\nobreakdash-acyclic. In order to prove that
$X$ is $H_2$\nobreakdash-acyclic, we need to show that for every $\sigma\colon P\to Q$ in $\calG$
and every $f\colon P\to RH_2X$ there is a morphism $g\colon Q\to RH_2X$ such that
$g\circ\sigma=f$. 

Write $X\cong\colim_{\calK}\,D$ for a diagram $D\colon\calK\to\calM$ where 
$\calK$ is $\lambda$\nobreakdash-filtered and $Dk$ is in $\calM_{\lambda}$ for all $k\in\calK$.
Then $H_1X\cong\colim_{\calK}\,(H_1\circ D)$ and $H_2X\cong\colim_{\calK}\,(H_2\circ D)$. 
Suppose given $f\colon P\to RH_2X$ for a morphism $\sigma\colon P\to Q$ in~$\calG$. 
Since $P$ is $\lambda$\nobreakdash-presentable, $f$ factors as
\[
\xymatrix{
P \ar[r]^-{f'} & RH_2Dk \ar[rr]^{RH_2\delta_k} & & RH_2X
}
\]
for some $k\in\calK$, where $\delta_k\colon Dk\to X$ denotes the corresponding cocone morphism. 
Thus, we may consider the set $\Trivializer_{H_2}(\sigma,Dk,f')$ in $\Ohkawa(H_2)$, which is nonempty
since $Dk\to *$ is in it, as $H_2(*)$ is contractible.

By assumption, $\Trivializer_{H_2}(\sigma,Dk,f')$ is then a member of $\Ohkawa(H_1)$, so there 
is an object $A\in\calM_{\lambda}$ and there are morphisms
$\tau\colon U\to V$ in $\calG$ and $u\colon U\to RH_1A$ such that 
\begin{equation}
\label{trivializers}
\Trivializer_{H_2}(\sigma,Dk,f')=\Trivializer_{H_1}(\tau,A,u).
\end{equation}
This forces, by definition, that $A=Dk$.

Since $H_1X$ is contractible, the morphism $RH_1X\to *$ is a trivial fibration
and hence there is a morphism $v\colon V\to RH_1X$ such that $v\circ\tau=RH_1\delta_k\circ u$.
Since $V$ is $\lambda$\nobreakdash-presentable, there is an object $k'\in\calK$
such that $v$ factors as
\[
\xymatrix{
V \ar[r]^-{w} & RH_1Dk' \ar[rr]^-{RH_1\delta_{k'}} & & RH_1X.
}
\]
Since $\calK$ is filtered,
there is an object $k''\in\calK$ together with morphisms $\alpha\colon k\to k''$ and $\beta\colon k'\to k''$.
Furthermore, since $U$ is $\lambda$\nobreakdash-presentable and
\[
RH_1\delta_{k''}\circ RH_1D\alpha\circ u = RH_1\delta_{k''}\circ RH_1D\beta\circ w\circ\tau,
\]
there is an object $k'''\in\calK$ and a morphism $\gamma\colon k''\to k'''$ such that the two composites
\[
\xymatrix{
U \ar[r]^-{u} & RH_1Dk \ar[rr]^-{RH_1D(\gamma\circ\alpha)} & & RH_1Dk'''
}
\]
and
\[
\xymatrix{
U \ar[r]^-{\tau} & V \ar[r]^-{w} & RH_1Dk' \ar[rr]^-{RH_1D(\gamma\circ\beta)} & & RH_1Dk'''
}
\]
coincide. Then $D(\gamma\circ\alpha)$ is in $\Trivializer_{H_1}(\tau,Dk,u)$
and therefore, by~\eqref{trivializers}, it is also in $\Trivializer_{H_2}(\sigma,Dk,f')$,
which means that the composite
\[
\xymatrix{
P \ar[r]^-{f'} & RH_2Dk \ar[rr]^-{RH_2D(\gamma\circ\alpha)} & & RH_2Dk'''
}
\]
factors through $\sigma\colon P\to Q$. Hence $f\colon P\to RH_2X$ also factors through $\sigma$ and this fact concludes the proof.
\end{proof}

\section{Consequences}
\label{consequences}

\begin{corollary}
\label{aset}
If a model category $\calM$ admits a $\lambda$\nobreakdash-combinatorial structure for a regular cardinal~$\lambda$, then there is only a set of distinct classes $\Acyclic(H)$ where $H$ runs over all functors 
$\calM\to\calM$ that preserve $\lambda$\nobreakdash-filtered colimits and such that
the terminal object is $H$\nobreakdash-acyclic.
\end{corollary}

\begin{proof}
Suppose that there is a proper class of functors $H_{i}$
preserving $\lambda$\nobreakdash-filtered colimits,
such that the classes $\Acyclic(H_{i})$ are all distinct and contain the terminal object. 
Then, by Theorem~\ref{mainthm}, after any choice of a $\lambda$\nobreakdash-combinatorial structure on $\calM$ the sets
$\Ohkawa(H_{i})$ will be distinct. This is impossible, since all sets $\Ohkawa(H_{i})$
are contained in the power set of the union of $\calM(A,B)$ for all $A,B\in\calM_{\lambda}$, where $\calM_{\lambda}$ denotes the chosen set of representatives of isomorphism classes of $\lambda$\nobreakdash-presentable objects in~$\calM$.
\end{proof}

Observe that this argument yields a bound on the cardinality of the set
of distinct classes $\Acyclic(H)$ for each sufficiently large regular cardinal~$\lambda$, namely
$2^{2^{\kappa}}$ where $\kappa$ is the cardinality of the set of all morphisms between objects of~$\calM_{\lambda}$. 

As pointed out in \cite{DP}, the cardinality of the set of homological acyclic classes in the homotopy category of spectra is bounded above by $2^{2^{\aleph_0}}$, since there are only countably many isomorphism classes of finite spectra.
Homological acyclic classes of spectra form a lattice, whose precise size is not known. Its cardinality is at least $2^{\aleph_0}$, since distinct sets of primes $J$ yield distinct acyclic classes represented by Moore spectra $M\mathbb{Z}[J^{-1}]$. 
Another set of distinct homological acyclic classes of spectra of cardinality $2^{\aleph_0}$ was displayed in~\cite[Lemma~3.4]{DP}, namely those represented by $\bigvee_{n\in A}K(n)$ for every subset $A$ of $\NN\cup\{\infty\}$. Lattices of homological acyclic classes have been calculated in several localized categories of spectra, including the harmonic category; see~\cite{W}.

\begin{corollary}
\label{bset}
If $\calM$ is a pointed combinatorial model category, 
then there is only a set of distinct classes $\Acyclic(H)$ where 
$H\colon\calM\to\calM$ has a right adjoint.
\end{corollary}

\begin{proof}
Left adjoints preserve all colimits and, in particular, the initial object (which is also terminal, since $\calM$ is pointed).
Hence, we may pick a regular cardinal $\lambda$ such that $\calM$ admits a $\lambda$\nobreakdash-combin\-atorial structure and the result follows from Corollary~\ref{aset}.
\end{proof}

Let $\calM$ be a monoidal model category in the sense of \cite[\S 4.2]{H}, so we tacitly assume that it is closed, but not necessarily symmetric.
For an object $E$ of~$\calM$, the \emph{Bousfield class} $\langle E\rangle$ is the class of all objects $X$
such that the derived tensor product of $E$ and $X$ is isomorphic to the terminal object in the homotopy category $\Ho(\calM)$. 
Thus, the following statement generalizes Ohkawa's theorem.

\begin{corollary}
\label{bousfield}
If $\calM$ is a pointed combinatorial monoidal model category,
then there is only a set of distinct Bousfield classes in~$\Ho(\calM)$.
\end{corollary}

\begin{proof}
Let $\lambda$ be a regular cardinal such that $\calM$ has a $\lambda$\nobreakdash-combinatorial structure and let
$Q$ be the chosen cofibrant replacement functor
that preserves $\lambda$\nobreakdash-filtered colimits on~$\calM$.
For each object~$E$, consider the functor $H_E\colon\calM\to\calM$ defined as $H_EX=QE\wedge QX$. Then $H_E$
preserves $\lambda$\nobreakdash-filtered colimits
for all~$E$, since the functor $QE\wedge (-)$ has a right adjoint ${\rm Hom}_{\ell}(QE,-)$ and hence it preserves all colimits, including the zero object.
Moreover, the Bousfield class $\langle E\rangle$ is equal to $\Acyclic(H_E)$, as $QE\wedge QX$ represents the derived tensor product of $E$ and~$X$.
Since, by Corollary~\ref{aset}, there is only a set of distinct
classes $\Acyclic(H)$ where $H$ preserves $\lambda$\nobreakdash-filtered colimits
and the zero object, the claim follows.
\end{proof}

\begin{corollary}
\label{derived}
For every commutative ring $R$ there is only a set of distinct Bousfield classes in the derived category $\calD(R)$.
\end{corollary}

\begin{proof}
For every ring $R$,
the category $\calD(R)$ is the homotopy category of the model category $\Ch(R)$ of unbounded chain complexes of $R$\nobreakdash-modules with the standard model structure \cite[Definition~2.3.3]{H}. This structure is combinatorial \cite[Theorem~2.3.11]{H} and it is symmetric monoidal if the ring $R$ is commutative \cite[Prop\-osi\-tion~4.2.13]{H}.
\end{proof}

According to \cite[IV.2]{EKMM} or \cite[Theorem~5.1.6]{SS2}, the category $\calD(R)$ is equivalent to the homotopy category of (strict) $HR$\nobreakdash-module spectra for each commutative ring $R$, where $HR$ denotes the Eilenberg--Mac\,Lane spectrum of ordinary cohomology with coefficients in~$R$. Thus, the following result extends Corollary~\ref{derived}. By a commutative ring spectrum we mean a commutative monoid in the category of symmetric spectra over simplicial sets~\cite{HSS}.

\begin{corollary}
\label{Emodules}
For every commutative ring spectrum $E$ there is only a set of distinct Bousfield classes in the homotopy category of $E$\nobreakdash-module spectra.
\end{corollary}

\begin{proof}
Modules over a commutative ring spectrum $E$ admit a symmetric monoidal model category structure which is combinatorial; see \cite[Theorem~4.1]{SS} for details.
\end{proof}

Let $S$ be a Noetherian scheme of finite Krull dimension and denote by ${\rm Sm}/S$ the category of smooth schemes of finite type over~$S$.
Let $\Mot_S$ be the category of pointed simplicial presheaves on ${\rm Sm}/S$, that is, contravariant functors from ${\rm Sm}/S$ to pointed simplicial sets.
Each pointed simplicial set is viewed as a constant presheaf, and each object of ${\rm Sm}/S$ is treated as a discrete simplicial presheaf via the Yoneda embedding, with an added disjoint basepoint.

Since ${\rm Sm}/S$ is equivalent to a small category, $\Mot_S$ is locally finitely presentable by \cite[Corollary~1.54]{AR}. Moreover, as shown in \cite[\S 2]{DRO} or \cite[Theorem~1.2]{J}, the Nisnevich topology on ${\rm Sm}/S$ endows $\Mot_S$ with a proper, cofibrantly generated, monoidal model category structure (with object\-wise smash product), whose associated homotopy category is equivalent to the pointed motivic homotopy category ${\rm H}_*(S)$ of Morel--Voevodsky \cite{MV, V} over the base scheme~$S$.

The category $\Mot_S$ can be stabilized into a monoidal stable model category by considering \emph{motivic symmetric spectra} with respect to the Thom space $T=\mathbb{A}_S^1/(\mathbb{A}_S^1-\{0\})$ of the trivial line bundle over~$S$ as in~\cite{J}, or \emph{motivic $\mathbb{S}$\nobreakdash-modules} as in~\cite{Hu}, or \emph{motivic functors} as in~\cite{DRO}.
All these stable model categories are Quillen equivalent, and their homotopy categories are equivalent to the stable motivic homotopy category~$\SH$. 

It is important to make a distinction between Bousfield classes and homological acyclic classes in the motivic context. Namely, if $E$ and $X$ are motivic spectra, the reduced $E$\nobreakdash-homology groups of $X$ are defined for $p,q\in\ZZ$~as
\[
E_{p,q}(X)=\pi_{p,q}(E\wedge X)=[S_s^{p-q}\wedge S_t^q,E\wedge X],
\] 
where $S_s^1$ is the simplicial circle $\Delta^1/\partial\Delta^1$ and $S^1_t$ is the algebraic circle $\mathbb{A}_S^1-\{0\}$, and
no notational distinction is made between a motivic space and its associated suspension spectrum. The \emph{homological acyclic class} of $E$ is the class of those $X$ such that $E_{p,q}(X)=0$ for all $p$ and~$q$, while the \emph{Bousfield class} of $E$ is the class of those $X$ such that $E\wedge X=0$ in~$\SH$.  As explained in \cite[\S 9]{DI} or \cite[\S 3.2]{J}, the latter condition is equivalent to $\pi_{p,q}(U_+\wedge E\wedge X)=0$ for all $p$ and $q$ and all smooth schemes $U$ of finite type over~$S$, where $U_+$ denotes the disjoint union of $U$ and~$S$.
Hence, $E\wedge X=0$ is a stronger statement than $E_{*,*}(X)=0$. Note, however, that if the homological acyclic classes of $E$ and $F$ coincide then their Bousfield classes coincide as well. 

As we next state, motivic Bousfield classes form a set. The same result for homological acyclic classes is proved in Corollary~\ref{motivic2}.

\begin{corollary}
\label{motivic1}
For each Noetherian scheme $S$ of finite Krull dimension there is only a set of distinct Bousfield classes in the stable motivic homotopy category $\SH$ with base scheme~$S$.
\end{corollary}
\begin{proof}
As shown in \cite{J}, the category of motivic symmetric spectra admits a proper, cofibrantly generated, monoidal model category structure whose homotopy category is equivalent to~$\SH$. 
Hence, Corollary~\ref{bousfield} applies.
\end{proof}

According to \cite[Theorem~13]{NS} or \cite[Proposition~5.5]{V}, the full subcategory of compact objects in $\SH$ is countable if ${\rm Sm}/S$ is countable (where a category is called \emph{countable} if it is equivalent to a category with only countably many morphisms). This implies that, if $S$ can be covered by affine open subsets ${\rm Spec}(R_i)$ where each ring $R_i$ is countable, then the cardinality of the lattice of Bousfield classes in $\SH$ is bounded above by $2^{2^{\aleph_0}}$.
This bound also follows from tensor triangulated category arguments; cf.\,\cite[Theorem~2.3]{IK}.

\begin{corollary}
\label{ostvaer}
There is only a set of distinct Bousfield classes in the derived category ${\rm DM}(k)$ of motives over any field $k$ of characteristic zero.
\end{corollary}

\begin{proof}
As shown in \cite[Theorem~1]{RO2},  the category ${\rm DM}(k)$ is equivalent to the homotopy category of modules over the commutative symmetric ring spectrum $M\ZZ$ that represents motivic cohomology for the given base field~$k$. According to \cite[Proposition~38]{RO2}, such modules form a symmetric monoidal model category. Since this model category is indeed combinatorial by \cite[Theorem~4.1]{SS}, we may use again Corollary~\ref{bousfield}.
\end{proof}

If $\calC$ and $\calD$ are any two categories and $\calD$ has a terminal object~$*$, then the \emph{kernel} of a functor $H\colon \calC\to\calD$ is the class of objects $X$ in $\calC$ such that $HX\cong *$.

Suppose that $\calC$ is locally $\lambda$\nobreakdash-presentable
and its terminal object is $\lambda$\nobreakdash-presentable.
Then, as mentioned in Section~\ref{prelims}, if we endow $\calC$ with the discrete model structure, the resulting model category has a $\lambda$\nobreakdash-comb\-in\-atorial structure.
For a functor
$H\colon\calC\to\calC$, the acyclic class $\Acyclic(H)$ is the kernel of~$H$.
Hence, Corollary~\ref{aset} specializes to the statement that, if $\lambda$
is a regular cardinal such that $\calC$ is locally $\lambda$\nobreakdash-presentable and its terminal object is $\lambda$\nobreakdash-presentable,
then there is only a set of distinct kernels of 
functors $\calC\to\calC$ preserving $\lambda$\nobreakdash-filtered colimits and the terminal object.
The following variant is more useful.

\begin{corollary}
\label{discrete}
Let $\calC$ and $\calD$ be locally $\lambda$\nobreakdash-presentable
categories, where $\lambda$ is a regular cardinal. Suppose that the terminal object of $\calC$ is $\lambda$\nobreakdash-pres\-ent\-able and $\calD$
has a zero object. Then there is only a set of distinct kernels of 
functors $H\colon\calC\to\calD$ that preserve $\lambda$\nobreakdash-filtered colimits and terminal objects.
\end{corollary}

\begin{proof}
Note that, since $\calD$ is locally $\lambda$\nobreakdash-presentable, an object $Y$ of $\calD$
is isomorphic to the zero object $0$ if and only if each morphism $P\to Y$
with $P\in\calD_{\lambda}$ factors through~$0$.
For each functor $H\colon\calC\to\calD$, consider the set $\Ohkawa(H)$
whose elements are the sets
\[
T_H(f)=\{t\in\calC(A, B) \mid \text{$B\in\calC_{\lambda}$ and $Ht\circ f$ factors through $0$}\},
\]
where $f$ runs over all morphisms $P\to HA$ in which $A\in\calC_{\lambda}$ and $P\in\calD_{\lambda}$.
Then it follows as in the proof of Theorem~\ref{mainthm} that
an equality $\Ohkawa(H_1)=\Ohkawa(H_2)$ implies that the kernels of $H_1$ and $H_2$ coincide,
if $H_1$ and $H_2$ preserve $\lambda$\nobreakdash-filtered colimits and terminal objects.
Since there is only a set of distinct sets~$\Ohkawa(H)$, the claim is proved.
\end{proof}

If $E_*$ denotes the reduced homology theory on pointed simplicial sets represented by a spectrum~$E$, then the condition $E_*(X)=0$ on a given $X$ is equivalent to $E\wedge\Sigma^{\infty}X=0$ in the homotopy category of spectra. Hence, it follows from Ohkawa's theorem that the collection of distinct homological acyclic classes of pointed simplicial sets is also a set.
This result can be inferred directly from Corollary~\ref{discrete} without passing to the category of spectra, since representable homology theories preserve $\aleph_0$\nobreakdash-filtered colimits if viewed as functors from pointed simplicial sets to graded abelian groups.

The same argument is valid in motivic homotopy theory:

\begin{corollary}
\label{motivic2}
There is only a set of distinct homological acyclic classes in the unstable motivic homotopy category and in the stable motivic homotopy category over any base scheme $S$.
\end{corollary}

\begin{proof}
This follows from Corollary~\ref{discrete}, both in the stable case and in the unstable case, by viewing $E_{*,*}$ as a functor to bigraded abelian groups for each motivic spectrum~$E$. This functor preserves $\aleph_0$\nobreakdash-filtered colimits since smashing with a cofibrant replacement of $E$ has a right adjoint and the circles $S^1_s$ and $S^1_t$ are finitely presentable.
\end{proof}

Note that Corollary~\ref{motivic1} also follows from Corollary~\ref{discrete} by letting $\pi_{*,*}$ take values in the category of presheaves of bigraded abelian groups on~${\rm Sm}/S$, which is locally finitely presentable by \cite[Corollary~1.54]{AR}.

\bigskip

\end{document}